\documentclass[11pt]{amsart}
\usepackage{fullpage}
\usepackage{amsfonts}
\usepackage{amsmath}
\usepackage{ bbold }
\usepackage{amssymb}
\theoremstyle{definition}
\newtheorem{definition}[subsection]{Definition}
\theoremstyle{plain}
\newtheorem{theorem}[subsection]{Theorem}
\newtheorem{proposition}[subsection]{Proposition}
\newtheorem{corollary}[subsection]{Corollary}
\newtheorem{lemma}[subsection]{Lemma}

\DeclareSymbolFont{bbold}{U}{bbold}{m}{n}
\DeclareSymbolFontAlphabet{\mathbbold}{bbold}
\newcommand\bbone{\mathbbold{1}}

\title{Representing interpolated free group factors \\ 
as group factors}
\author{Sorin Popa and Dimitri Shlyakhtenko}
\email{popa@math.ucla.edu, shlyakht@math.ucla.edu}
\address{Department of Mathematics, UCLA, Los Angeles, CA 90095}
\thanks{Research supported by NSF grants DMS-1700344 and DMS-1500035}
\begin{document}
\begin{abstract} We construct a one parameter family of ICC groups $\{G_t\}_{t>1}$, with the property that 
the group  factor $L(G_t)$ is isomorphic to the interpolated free group factor $L(\mathbb F_t):=L(\mathbb{ F}_2)^{1/\sqrt{t-1}}$, $\forall t$. Moreover, the groups $G_t$ have fixed cost $t$, are strongly treeable and freely generate any treeable ergodic equivalence relation of same cost.  
\end{abstract}
\maketitle

\section{Introduction.}

Groups and their measure preserving actions on probability spaces 
give rise to two remarkable classes of II$_1$ factors: the {\it group factors}, 
obtained from the left regular representation 
of infinite conjugacy class (ICC) groups $G$ and denoted $L(G)$ ([MvN43]); the {\it group measure space factors}, 
constructed from free ergodic pmp actions of countable groups   
$G \curvearrowright X$, denoted $L^\infty(X)\rtimes G$ ([MvN36]). 

Another important source of II$_1$ factors is the  {\it amplification} $M^t$ of a given II$_1$ factor $M$ by positive real numbers $t>0$  
([MvN43]). Amplifications satisfy $(M^t)^s=M^{ts}$,  and if $t$ is an integer then $M^t$ is the algebra of $t \times t$ matrices with entries in $M$, 
while if $t\leq 1$ then $M^t \simeq pMp$, with $p$ a projection of trace $t$ in $M$. 

There has been much effort in 
understanding the structure and classification of these factors in terms of their ``building data''. 
A significant part of this work concentrated  on W$^*$-{\it rigidity} phenomena, 
which aim at recovering $G$, $G\curvearrowright X$, or the amplifying $t>0$, from  the isomorphism class of the associated II$_1$ factor 
(i.e., its W$^*$-{\it equivalence} class), 
or proving unique decomposition,  indecomposability and non-embeddability results (see e.g., \cite{popa:strongrig1}).  

At the opposite end, establishing isomorphism (or mere embedability) between II$_1$ factors arising from different groups (resp. group actions) 
is equally interesting.   It has already been shown in ([MvN43]) 
that all II$_1$ factors of the form $L(G), L^\infty(X)\rtimes G$ with $G$ locally finite are isomorphic to the hyperfinite II$_1$ factor $R$, 
and that $R^t\simeq R$, $\forall t>0$. Connes fundamental theorem on the uniqueness of the amenable amenable II$_1$ factor ([C75]), shows  
that the same holds true for all amenable groups $G$ and that  all II$_1$ factors embeddable into $R$ are isomorphic to $R$, with  
$L(G)\hookrightarrow R$ iff $G$ is amenable.  
Thus, all amenable ICC groups $H_i$ are W$^*$-equivalent and if $G$ is any ICC group then 
$G \times H_i$ are W$^*$-equivalent as well.  

Free probability theory allowed to  deduce that any free product of two or more infinite amenable groups 
$*_{i\in I} H_i$ is W$^*$-equivalent to $\mathbb F_{|I|}$ (\cite{dykema:interpolated}).  It also allowed to represent amplifications  
of free group factors as group factors in unexpected ways:  
Voiculescu's amplification formula \cite{voiculescu:89} showed that $L(\mathbb F_n)^{1/k}\simeq L(\mathbb F_{k^2(n-1)+1})$, $\forall n, k\geq 2$, and  
Dykema  \cite{dykema:interpolateddef},  Radulescu \cite{radulescu:inv} 
extended this to show that $L(\mathbb F_n)^t=L(\mathbb F_m)^s$ whenever $(n-1)/t^2=(m-1)/s^2$. So all amplifications of factors from finite rank free groups 
are among the {\it interpolated free group factors} 
$L(\mathbb F_t):=L(\mathbb{ F}_2)^{1/\sqrt{t-1}}$. 

The same  methods made it possible to identify some of the interpolated free group factors $L(\mathbb F_t)$ 
as group factors  for non-integer $t$, showing for instance that $L(PSL(2,\mathbb Z))=L(\mathbb F_{7/6})$ (\cite{dykema:interpolated}, \cite{harpe-picou:1993}).  
However, the problem of representing all II$_1$ factors $L(\mathbb F_t)$ as group factors remained open. We solve this problem here by proving: 

\begin{theorem}  \label{mainthrm} For any $t > 1$ there exists a countable ICC group $G_t$ with the property that $L(G_t)$ is isomorphic to the Dykema-Radulescu 
interpolated free group factor $L(\mathbb{F}_t)$.   Thus, 
$L(G_t)$ satisfy the amplification formula $L(G_t)^s\simeq L(G_{1+(t-1)/s^2}$.  
The groups $G_t$ can be chosen to be strongly treeable of fixed cost $t$. Moreover,  given any free ergodic pmp action $G_t \curvearrowright X$ there exists a free action $G_{1+(t-1)/s} \curvearrowright Y$ such that $(L^\infty(X)\rtimes G_t)^s\simeq L^\infty(Y)\rtimes G_{1+(t-1)/s}$, $\forall t>1, s>0$. \end{theorem}

The {\it strong treeability} and cost of the groups $G_t$ are  in the sense of \cite{gaboriau:cost}. In other words, the groups $G_t$ have the property that every free ergodic pmp action of $G_t \curvearrowright X$ has treeable orbits and cost $t$; in particular, the resulting equivalence relation satisfies  
$\mathcal R_{G_t}=(\mathcal R_{\mathbb F_2})^{1/\sqrt{t-1}}$,  for some free ergodic $\mathbb F_2\curvearrowright X$. The above result also shows that $G_t$ 
satisfy  another property, that we call {\it OE-freeness}: 
any treeable ergodic equivalence relation $R$ of cost $t$ is induced by an action of $G_t$. Thus, our groups $G_t$ are both W$^*$-{\it free} (a terminology we'll use to designate an ICC 
group whose  II$_1$ factor is isomorphic to an interpolated free group factor) and OE-free.  

Examples of OE-free groups $H_t$ with cost $t$, providing representations 
of $(\mathcal R_{\mathbb F_2})^{1/\sqrt{t-1}}$ as  $\mathcal R_{H_t}$, $\forall t>1$, 
have already been constructed by Hjorth in  \cite{hjorth:costAttained}. 
Our groups $G_t$ are in fact 
a modified version of the $H_t$ in \cite{hjorth:costAttained}, being obtained by repeated amalgamated free products with groups of the form $Q_n \times \mathbb F_{k_n}$,  
where $Q_n$ is an increasing sequence of finite groups such that $\cup_n Q_n$ is ICC,  the amalgamation is over $Q_{n-1}\subset Q_n$, and $k_n\geq 2$ are integers of appropriate size.

The proof that $L(G_t)\simeq L(\mathbb F_t)$, which we do in Section 2, utilizes operator valued semicircular systems, that enables the identification of such amalgamated free products of groups  
with appropriate models of the interpolated free group factors in (\cite{shlyakht:amalg,shlyakht:Avalued}, \cite{dykema:interpolated,dykema:interpolateddef,radulescu:inv}). 
The proof of the treeability, which we do in Section 3, is a slight extension of the argument in \cite{hjorth:costAttained}. 
Section 4 contains a number of remarks and open problems, and lists several properties of groups $G_t$, such as 
Cartan rigidity (\cite{popa-vaes:F2}) and  the universal vanishing cohomology property $\mathcal V\mathcal C$ 
(\cite{popa:VC}). 

\subsection*{Acknowledgements} This work was carried out in part during the ``Quantitive Linear Algebra'' program at the Institute for Pure and Applied Mathematics.  S.P. would like to acknowledge support from IPAM and the Simons Foundation as the Simons Participant.

\section{A family of groups generating interpolated free group factors.}

In this section we begin by introducing a certain family of groups $G=G(\{Q_i\},\{k_i\})$ and study their $W^*$-equivalence class. The main result of this section states that, under certain assumptions, the group von Neumann algebra of such a group $G$ is isomorphic to an interpolated free group factor, a property that we call W$^*$-{\it freeness}.  To prove this result, we first give a description of this group von Neumann algebra in terms of operator-valued semicircular systems.  

\begin{definition} \label{sect:construct}
Suppose that $\{e\}=Q_0 \subset Q_1 \subset Q_2 \subset \cdots$ is an increasing sequence of finite groups (so that $|Q_n|\to\infty$), and suppose that $k_0, k_1,\dots$ is a sequence of non-negative integers.   Let $H_n = Q_n \times \mathbb{F}_{k_n}$.  
Define $G_n$ recursively by:  $G_0 = H_0$, $G_{n+1} = G_n *_{Q_n} H_{n+1}.$  Finally, define  $G=G(\{Q_i\},\{k_i\})$ by $G = \bigcup_n G_n.$  
\end{definition}

The construction of the  groups $G(\{Q_i\},\{k_i\})$ is a slight generalization of the construction considered by G. Hjorth \cite{hjorth:costAttained}, which correspond in our notation to the groups 
 $G(\{\mathbb{Z}/2\mathbb{Z})^n\},\{k_n\})$.

To fix notations, let  $A_n = L(Q_n)$ and  $C_n = L(G_n)$ (taken with their group traces).  Then $C_0 \cong L(\mathbb{F}_{k_0})$ and $C_{n+1} \cong C_n *_{A_n} (A_{n+1} \otimes L(\mathbb{F}_{k_{n+1}}))$.  Let $C=W^*\left(\bigcup_n C_n \right)$; then one easily sees that $C=L(G)$. 

We now give a brief review of operator valued semicircular systems; the reader is referred to \cite{shlyakht:amalg,shlyakht:Avalued} for more details and proofs.

\subsection{Operator-valued free semicircular variables}\label{sect:OA-valued}

A family of {\em $D$-valued semicircular operators with covariances $\eta_1,\dots,\eta_k$} is associated to a family of completely positive maps $\eta_j : D\to D$ defined on a tracial von Neumann algebra $(D,\tau_D)$ and satisfying the self-adjointedness condition  $\tau_D(x\eta_j(y))=\tau_D(\eta_j(x)y)$ for all $x,y\in D$. This family of operators has the following properties:

(a) The von Neumann algebra $M:=W^*(D,X_1,\dots,X_k)$ admits a trace $\tau$ extending $\tau_D$.  Moreover, there exists a $D$-valued conditional expectation $E:M\to D$ so that $\tau_D\circ E=\tau$ and so that
$$
E(X_i y X_j)=\delta_{i=j} \eta_j(y),\qquad \forall y\in D.
$$

(b) For $j=1,\dots,k_n$, the algebras $W^*(D,X_j)$ are free with amalgamation over $D$.

(c) If $\eta_j(1)=1$, then $(W^*(X_j),\tau)\cong (L^\infty[-\sqrt{2},\sqrt{2}], \mu)$  where $\mu$ corresponds to integration with respect to the semicircular measure $\pi^{-1} \sqrt{2-t^2} dt$.

(d) More generally, suppose that for some unital subalgebra $D_0\subset D$, $\eta_j(D_0)\subset D_0$ for each $j$. Then $X_1,\dots,X_n$ also form a $D_0$-valued semicircular system with covariances $\eta_j' = \eta_j \Big|_{D_0}$.

(e) Suppose that $D_0\subset D_1 \subset D$ are a unital subalgebras and suppose that $\eta_j:D\to D $ is the $\tau_D$-preserving conditional expectation onto $D_0$.  Then $X_j$ is free from $D$ with amalgamation over $D_1$.  Furthermore, $W^*(D_0,X_j) =  D_0 \otimes W^*(X_j)$.

(f) Given $\eta_i$, consider the $D,D$-bimodule $H_i$ obtained from $D\otimes D$ after separation and completion with respect to the inner product $\langle a'\otimes b', a\otimes b\rangle = \tau (b'\eta_i (a'a^*)b^*)$, and denote by $K_i\subset H_i$ the closure of the real subspace $\{\sum a_i \otimes b_i^* +b_i\otimes a_i^*:a_i, b_i\in D\}$.  Then $K_i$ is a Jordan bimodule over $D$: for any $a,b\in D$ and $h\in K_i$, $ahb + b^* h a^* \in K_i$.  
The von Neumann algebra $W^*(D,X_1,X_2,\dots)$ depends up to isomorphism only on the isomorphism class of the Jordan subbimodule $\bigoplus_i K_i \subset \bigoplus_i H_i$.

\subsection{A description of the algebra $C_n$}
Returning now to the algebras $C_n$, for $j\leq n$ let  $E^n_j : A_n \to A_n$ be the trace-preserving conditional expectation onto $A_j \subset A_n$.  In particular, $E^n_0$ is the trace on $A_n$.  Furthermore, let $E_n : A \to A_n$ be the trace-preserving conditional expectation. 

\begin{lemma} For each $i\in \{1,\dots,n\}$ and each $j\in \{1,\dots,k_n\}$ let $\eta_{ij}=E_i$.  Let 
$\{ X_{ij} : i\in  \{1,\dots,n\}, j\in \{1,\dots,k_n\}\}$ be a family of $A$-valued semicircular variables with covariance $\eta_{ij}$.   Then  $C_n = W^*(A_n, X_{ij} : 0\leq i\leq n, 0\leq j\leq k_n)$.  
\end{lemma}

\begin{proof}
By \S\ref{sect:OA-valued}d, since $\eta_{ij}(A_n)\subset A_n$, for each fixed $i$, $X_{ij}$ are actually $A_n$-valued semicircular variables with covariances $E^n_j$.    

We now proceed by induction.  If $n=0$, then by \S\ref{sect:OA-valued}b, $X_{01},\dots,X_{0k_0}$ are free (with amalgamation over $\mathbb{C}1$) and since by \S\ref{sect:OA-valued}c, $W^*(X_{0j})\cong L(\mathbb{Z})$, this $k_0$-tuple  generates $C_0\cong L(\mathbb{F}_{k_0})$.  

Supposing that $W^*(A_n, X_{ij} : 0\leq i\leq n, 0\leq j\leq k_n)\cong C_n$, we note that because if $k\leq n$ 
$E^{n+1}_k$ is the conditional expectation onto $A_{k}\subset A_{n+1}$ we deduce from \S\ref{sect:OA-valued}e that $X_{ij}$ for $i\leq n$ are free with amalgamation over $A_n$ from $A_{n+1}$.  Moreover, by \S\ref{sect:OA-valued}b, $\{X_{(n+1) j}: 1\leq j\leq k_{n+1}\}$ are free from $W^*(A_{n+1}, X_{ij} : 0\leq i\leq n, 0\leq j\leq k_n)\cong C_n$ over $A_{n+1}$.  Thus by the inductive assumption and by \S\ref{sect:OA-valued}e,
\begin{multline*}
W^*(A_n, X_{ij} : 0\leq i\leq n+1, 0\leq j\leq k_n) \\
\cong W^*(A_n, X_{ij} : 0\leq i\leq n, 0\leq j\leq k_n)*_{A_n} 
W^*(A_{n+1},X_{(n+1) j}: 1\leq j\leq k_{n+1}) \\
 \cong C_n *_{A_n} (A_{n+1}\otimes L(\mathbb{F}_{k_{n+1}}))
 = C_{n+1}.
\end{multline*}  This concludes the proof. 
\end{proof}

\begin{corollary}  $C\cong W^*(A, X_{ij} : i\in \mathbb{N}\cup\{0\}, 0\leq j\leq k_n)$.
\end{corollary}

\subsection{Isomorphism class of $C$ when $Q$ is ICC}
Let $R$ be a II$_1$ factor and let $\tau$ be the unique trace on $R$ (for our purposes, we may assume that $R$ is the hyperfinite II$_1$ factor $R=L(Q)$).  Assume that $D\subset R$ is a unital finite-dimensional subalgebra.  Let $R_1$ be the Jones basic construction for $D\subset R$.  Denote by $H$ the $R,R$-bimodule $L^2(R)\otimes_D L^2(R)\cong L^2(R_1)$, and by $K$ the Jordan sub-bimodule  $K=\{x\in L^2(R_1):x=x^*\}$.  Because $D$ is finite-dimensional, $H\subset L^2(R)\otimes L^2(R)$. 
Since $L^2(R)\otimes_D L^2(R) = (L^2(R)\otimes L^2(R)) \otimes_{D'\otimes D'} L^2(D)$ we find from invariance of Murray-von Neumann dimension under induction that 
$$\dim_{L(R)\bar\otimes L(R)}H = \dim_{D\otimes D} L^2(D).$$ 
Let us write 
$$
b_0(D)=\dim_{D\otimes D} L^2(D).
$$  This is the zeroth $L^2$-Betti number of $D$ in the sense of \cite{cs:l2betti}.
For $D$ a matrix algebra of size $m\times m$, $b_0(D) = {m^{-2}}$. If $D = D_1 \oplus D_2\oplus \cdots \oplus D_n$ and if $p_i\in D_i$ denotes the unit of $D_i$, then $b_0(D) = \sum_i (\tau(p_i))^2 b_0(D_i)$.

\begin{proposition}  \label{prop:isomClassR} Suppose that $D_i \subset R$, $i\in \mathbb{N}\cup\{0\}$ are finite-dimensional subalgebras, and denote by $E_j$ the conditional expectation onto $D_j$.  Let $k_n \in \mathbb{N}$ be fixed, and suppose that $(\eta_{ij}: i\in \mathbb{N}\cup\{0\}, j\in\{1,\dots,k_n\})$ are completely positive maps given by $\eta_{ij} = E_i$.  Let $X_{ij}$ be the associated $R$-valued semicircular system. Then $$W^*(R,X_{ij}: i\in \mathbb{N}\cup\{0\},j\in \{1,\dots,k_n\})\cong L(\mathbb{F}_t)$$ with $t= 1 + \sum_n k_n b_0 (D_n)$.
\end{proposition}

\begin{proof}
For a subalgebra $D\subset R$, let $H$ and $K$ be as before: $H=L^2(R)\otimes_D L^2(R)\cong L^2(R_1)$ and $K=\{x\in L^2(R_1):x=x^*\}$.  
It is a well-known folklore result that the inclusion $H\subset L^2(R)\otimes L^2(R)$ 
 can be chosen in such a way that $K$ is isomorphic to $L^2(R)p \otimes p L^2(R)$ for a projection $p\in R$ with trace determined by  
$$\tau(p)^2 = \dim_{L(R)\bar\otimes L(R)} H =  b_0(D).$$   

Furthermore, if $D_i \subset R$ are subalgebras, $E_i$ is the conditional expectation on $D_i$ and $H_i$, $K_i$ the associated bimodule and Jordan bimodule, and if $k_n$ are given integers, then 
\begin{equation}\label{eqn:JordanIsom}
\bigoplus_n (K_n)^{\oplus k_n} \cong (L^2(R)\otimes L^2(R))^{\oplus N} \oplus L^2(R)p \otimes p L^2(R)
\end{equation} for some integer $N\in \mathbb{N}\cup \{0,+\infty\}$ and projection $p\in R$ satisfying $\tau(p)^2 + N = \sum_n k_n b_0(D_n)$ (let us agree to choose $p=0$ if $N=\infty$).

By \S\ref{sect:OA-valued}f if $X_{ij}$ are $R$-valued semicircular operators of covariance $\eta_{ij}=E_i$, then $W^*(R,X_{ij}: i\in \mathbb{N}\cup\{0\},j\in \{1,\dots,k_n\})$ depends only on the Jordan bimodule $\bigoplus_{nj} (K_{n})^{\bigoplus k_n}$, and thus only on the number $\sum_n k_n b_0(D_n)$.  

Moreover, choosing the isomorphism class of $\bigoplus_i K_i$ as in \eqref{eqn:JordanIsom} we see that 
$$W^*(R,X_1,X_2,\dots)\cong W^*(R,pY_0p,Y_1,\dots,Y_N)$$ where $Y_0,Y_1,\dots,Y_N$ are $R$-valued semicircular variables so that the variance of $Y_j$, $j=0,\dots,N$ is $x\mapsto \tau(x)$ (which implies that the variance of $pY_0p$, which is also $R$-semircircular, is $x\mapsto p\tau(p xp )$).  But this means that $Y_0, Y_1,\dots$ are a free semicircular family, free from $R$.  Thus by  the results in \cite{dykema:interpolated,dykema:interpolateddef,radulescu:inv}, see also Proposition~5.4 in \cite{shlyakht:amalg}, $W^*(R,X_1,X_2,\dots)$  is isomorphic to the Dykema-R\u{a}dulescu interpolated free group factor $L(\mathbb{F}_{t})$ where   $t = 1+N+\tau(p)^2 = 1 + \sum_n k_n b_0(D_n)$.  
\end{proof}

\begin{corollary}\label{cor:groupCR} Assume that the group $Q$ is ICC. Then $W^*(G)\cong L(\mathbb{F}_t)$ where $t = 1 + \sum_n k_n {|Q_n|^{-1}}$.
\end{corollary}
\begin{proof}
If $D=L(T)$ for some finite group $T$, then $D=\bigoplus_\pi D_\pi$ where the direct sum is over all irreducible representations $\pi$ and $D_\pi$ is a matrix algebra of size $\dim\pi \times \dim\pi$.  Denote by $p_\pi$ the minimal central projection of $D$ corresponding to the summand $D_\pi$ (i.e., the unit of $D_\pi$).  Then the trace $\tau$ on $D$ is determined by $$\tau(p_\pi) = \frac{(\dim \pi)^2}{|T|}.$$ 
Thus $$b_0(D) = \sum_\pi \frac{(\dim \pi)^4}{|T|^2 (\dim{\pi})^2} = \sum_\pi \frac{(\dim\pi)^2}{|T|^2 }
 = \frac{1}{|T|^2}\sum_\pi (\dim \pi)^2 = \frac{|T|}{|T|^2} = \frac{1}{|T|}.$$  (Rather than this computation, we could have just used \cite{cs:l2betti}  to conclude that $b_0(D)=b_0(L(T))=b_0(T)$, the zeroth $L^2$-Betti number of $T$).  Applying this fact in the context of Proposition~\ref{prop:isomClassR} gives the statement of the corollary.
\end{proof}

\begin{theorem}  \label{thrm:groupAlg} Let $t\in (1,+\infty]$ be arbitrary. Then there exists a countable discrete group $G_t$ so that $L(G_t)$ is isomorphic to the Dykema-Radulescu interpolated free group factor $L(\mathbb{F}_t)$.  Thus, the group factors $L(G_t)$ satisfy the amplification formula $L(G_t)^s\simeq L(G_{1+(t-1)/s^2})$. 
\end{theorem}
\begin{proof}
Let $Q_n = S_n$, the symmetric group on $n$ letters.  Then $Q=S_\infty$ is ICC. Moreover, $b_0 (A_n) = 1/|Q_n|=1/n!$.  Since $\sum_n \frac{1}{n!} <\infty$, for any $t\in (1,+\infty]$ we can choose  
integers $k_n\in \mathbb{N}\cup\{0\}$ in such a way that $$t = 1+\sum_n \frac{k_n}{n!}.$$  Let $G$ be the group constructed in \S\ref{sect:construct}.  Then by Corollary~\ref{cor:groupCR}, $$L(G)\cong C\cong L(\mathbb{F}_t).$$  This completes the proof.
\end{proof}

The following corollary follows from the fact that interpolated free groups factors $L(\mathbb{F}_t)$ for all $t\in (1,+\infty)$ are stably isomorphic: 

\begin{corollary}  \label{cor:WE} Let $Q_0 \subset Q_1\subset\cdots$, $Q_0'\subset Q_1'\subset\cdots$ be  sequences of groups and $k_0, k_1,\dots$, $k_0',k_1',\dots$ be  sequences of integers, and let $G=G(\{Q_i\},\{k_i\})$, $G'=G(\{Q'_i,k'_i\})$.  Assume that $Q=\bigcup_n Q_n$ and $Q'=\bigcup_n Q_n'$ are both ICC.  Let $d = 1+\sum k_n |Q_n|^{-1}$, $d' = \sum_n k_n' |Q'_n|^{-1}$.  If $d,d'$ are both simultaneously finite or simultaneously infinite, then $L(G)$ and $L(G')$ are stably isomorphic.  If $d=d'$, then $L(G)\cong L(G')$.  
\end{corollary}

\subsection{Remarks} \label{remark:Gt} $1^\circ$ Note that in Theorem~\ref{thrm:groupAlg} we could have set $G_t$ to be any group $G(\{Q_i\},\{k_i\})$ as long as $Q=\bigcup Q_i$ is ICC and $t=1 + \sum k_n |Q_n|^{-1}$.  We will use the notation $G_t$ somewhat ambiguously for any such group.

$2^\circ$ If one drops the ICC condition on the group $Q$, the von Neumann algebra $L(G(\{Q_i\},\{k_i\}))$ may fail to be a factor. Indeed, suppose that $Q_n = (\mathbb{Z}/2\mathbb{Z})^n$, as in Hjorth's original example, and $k_0=0$.  Then $Q_1$ clearly belongs to the center of $G$.   

$3^\circ$ On the other hand, suppose that $Q_n$ are arbitrary (e.g., $Q_n = (\mathbb{Z}/2\mathbb{Z})^n$) and $k_0 \neq 0$.  Then we can repeat our previous argument, replacing the algebra $R$ in the previous sections by the factor $W^*(G_0, Q)\cong L(G_0)* L(Q)\cong L(\mathbb{F}_{k_0+1})$, to deduce that $L(G(\{Q_i\},\{k_i\}))\cong L(\mathbb{F}_t)$ with $t=1 + \sum_n k_n {|Q_n|^{-1}}$.  This gives another family of groups generating interpolated free group factors $L(\mathbb{F}_t)$, but for $t\geq 2$.

$4^\circ$ We expect that $L(G(\{Q_i\},\{k_i\}))$ must be an interpolated free group factor whenever it is a factor, but we were not able to prove this.

\section{Measure Equivalence theory of the groups $G(\{Q_i\},\{k_i\})$.}

Let us first recall a few notions about treeability (see \cite{gaboriau:cost,kechris-miller:OE}).  An equivalence relation $R$ is {\em treeable} iff it has a graphing $\Gamma$ which is a treeing.  Equivalently, $R$ is the free product of a family of finite and hyperfinite equivalence relations.  

A group $G$ is said to be {\em treeable} if there exist a free probability measure-preserving (free pmp) action of $G$ so that the associated orbit equivalence relation $R_G$ is treeable.  A group is {\em strongly treeable} if {\em every} free pmp action of $G$ is treeable.  

We will say that a group $G$ is {\em OE-free} if  for some number $t\in [1,+\infty]$, any ergodic treeable equivalence relation of cost $t$ arises from a free pmp action of $G$ (in which case it necessarily holds that the first $L^2$ Betti number of $G$ is $\beta_1^{(2)}(G)=t+1$).  Clearly, OE-free $\implies$ treeable  and strogly treeable $\implies$ treeable; it's not known if any of these implications can be reversed when $t\neq 1$.  

The first examples of OE-free groups were found by Hjorth in \cite{hjorth:costAttained}.  He showed that this property holds for free groups $\mathbb{F}_n$ as well as (in our notation) the groups $G(\{\mathbb{Z}/2\mathbb{Z})^n\},\{k_n\})$.   The main result of this section is that  groups $G(\{Q_i\},\{k_i\})$ are OE-free and strongly treeable.  Strong treeability (as well as the formula for the cost) follows from Gaboriau's results \cite{gaboriau:cost}; however, we give the full argument, since we anyways need much of the set-up to give a proof of OE-freeness.  The proof of the latter property is a slight generalization of Hjorth's argument \cite{hjorth:costAttained}.  Together with the results of Section 2, the results of this section complete the proof of Theorem~\ref{mainthrm}. 

\subsection{Equivalence relations $R_Q$, $R_G$ and $R_n$}
Assume that $\alpha$ is a free measure preserving action of $G$ on a probability space $(X,\mu)$.  Let us denote by $R_G$ the equivalence relation of being in the same orbit of this action.  Denote by $R_Q$ the equivalence relation of being in the same orbit under the restriction of $\alpha$ to $Q$.  Then $R_Q$ is a sub-equivalence relation of $R_G$.  Finally, let $R_n$ be the equivalence relation of being in the same orbit of the restriction of $\alpha$ to $Q_n$. Then $R_n\subset R_{n+1}$ for all $n$, and $R_Q = \bigcup_n R_n$.  Clearly, each $R_n$ is an equivalence relation all of whose orbits are finite, consisting of $|Q_n|$ elements.  

\begin{lemma} \label{lem:partitions}
There exist partitions 
 $P_n = (F(j_1,\dots,j_n): j_k \in \{1,\dots,q_{k}\})$ of the set $X$ so that $P_{n+1}$ is a refinement of $P_n$ and $P_n$ is a fundamental domain for $R_n$: every orbit of $R_n$ meets each element of the partition $P$ exactly once.  In other words, for any $x\neq y$, $x\sim_{R_n} y$ implies that $x$ and $y$ are points in different elements of the partition $P_n$, and for each $A,B\in P_n$, $A\neq B$, and any $x\in A$, there exists a unique $y\in B$ so that $x\sim_{R_n} y$.  
\end{lemma}

\begin{proof}
Denote by $X_n$ the quotient $X_n/R_n$ (in particular, $X_0 = X$).  Let $\pi_n : X\to X_n$ be the quotient map, and let $\mu_n$ be the push-forward of the measure $\mu$ by $\pi_n$.  Since $R_n\subset R_{n+1}$, the map assigning to each $R_n$-orbit the unique $R_{n+1}$ orbit containing it gives a map $\rho_n : X_n\to X_{n+1}$; by construction, $\rho_n\circ \pi_n = \pi_{n+1}$.  

Let $q_{n+1} = |Q_{n+1} / Q_n|$.  For each $n$, we can choose measurable maps $\sigma_{1,n+1},\dots,\sigma_{q_{n+1},n+1} : X_{n+1}\to X_n$ in such a way that (i) each $\sigma_{j,n+1}$ is a cross-section to $\rho_n$, meaning that $\rho_n\circ \sigma_{j,n+1}(x) = x$ for a.e. $x\in X_{n+1}$ and (ii) the sets $\sigma_{j,n+1} (X_{n+1})$, $j=1,\dots,q_{n+1}$, are disjoint (modulo null sets) and cover $X_{n}$ (modulo null sets), i.e., form a partition of $X_n$.  

For each $n$ and each $j_1,\dots,j_n$ so that $j_k\in \{1,\dots,q_k\}$, define subsets
$$F(j_1,\dots,j_n) = \sigma_{j_1,1} \circ \sigma_{j_2,2} \circ \cdots \circ \sigma_{j_n, n} (X_n) \subset X.$$  
By construction, $$F(j_1,\dots,j_{n+1})\subset F(j_1,\dots,j_n)$$ since $ \sigma_{j_1,1} \circ \sigma_{j_2,2} \circ \cdots \circ \sigma_{j_n, n} \circ \sigma_{j_{n+1},n+1} (X_{n+1})\subset  \sigma_{j_1,1} \circ \sigma_{j_2,2} \circ \cdots \circ \sigma_{j_n, n} (X_n)$.  

We claim that  $(F(j_1,\dots,j_n): j_k \in \{1 ,\dots,q_k\})$ form a partition of $X$.  We check this by induction.  If $n=1$, then $F(j_k) = \sigma_{j,1}(X_1)$ and these form a partition of $X_0$.  Having checked the inductive step, we note that $(\sigma_{k,n+1}(X_{n+1}):k=1,\dots,q_{n+1})$ form a partition of $X_{n}$, and so  $\sigma_{j_1,1} \circ \sigma_{j_2,2} \circ \cdots \circ \sigma_{j_n, n} (X_n)$ is the disjoint union of $ 
  \sigma_{j_1,1} \circ \sigma_{j_2,2} \circ \cdots \circ \sigma_{j_n, n} \circ \sigma_{k,n+1}(X_{n+1})$ as $k$ ranges from $1$ to $q_{n+1}$. This completes the inductive step.
\end{proof}  

\subsection{A graphing for $R_G$}
Define now $\theta: X\to X$ to be the odometer built on the partitions $P_n$. More precisely, note that for almost all $x\in X$, we can choose the smallest $n$ so that there exists an $r\leq n$ and $j_1,\dots,j_n$ with $j_r<q_r$ and $x\in F(j_1,\dots,j_n)$.   Choose $r$ to be the smallest possible integer with this property. We define $\theta(x)$ to be the unique $y\in F(j_1,\dots,j_r+1,\dots,j_n)$ so that $x\sim_{R_n} y$.   It is clear from the construction that the equivalence relation generated by $\theta$ is exactly $R_Q = \bigcup R_n$.

Recall that $G = \bigcup_n G_n$ with $G_{n+1} = G_{n} *_{Q_n} (\mathbb{F}_{k_{n+1}} \otimes  Q_{n+1})$.  Let us denote by $g_{n+1,j}$, $j=1,\dots,k_{n+1}$ the generators of $\mathbb{F}_{k_{n+1}}$.  Since $\theta$ generates $R_Q$, it follows that the set
$$ \Gamma = \{ \theta, \alpha(g_{n,j}) : n=0,1,\dots, 1\leq j\leq k_n\}$$
is a graphing for $R_G$.  

Let $\phi_{n,j}$ denote the restriction of $\alpha(g_{n,j})$ to the set $F({\bbone}_n)$, where we write ${\bbone}_n$ for the $n$-tuple consisting of all $1$'s.  

\begin{lemma} \label{lem:graphing} The set
$$\Gamma '=  \{ \theta, \phi_{n,j}  : n=0,1,\dots, 1\leq j\leq k_n\}$$ is a graphing for $R_G$.  
\end{lemma}
\begin{proof}Indeed, denote by $R\subset R_G$ the equivalence relation generated by $\Gamma'$.  To see that $R = R_G$, suppose that $x,y\in X$ and for some $n$ and $j$, $y=\alpha(g_{n,j})(x)$.  Then almost surely, $x\in F(j_1,\dots,j_n)$ for some $j_1,\dots,j_n$, which means that for some $h\in Q_n$, $\alpha(h)(x)\in F(\mathbb{1}_n)$.  Since $Q_n$ commutes with $g_{n,j}$, it follows that 
\begin{equation*}
 y = \alpha(g_{n,j}) \alpha(h)^{-1} \alpha(h) (x) 
  = \alpha(h)^{-1} \alpha(g_{n,j}) \alpha(h) (x) 
  = \alpha(h)^{-1} \theta(g_{n,j}) \alpha(h) (x).
 \end{equation*}
 It follows that if we set $x'=\alpha(h)(x)$, $y'=\alpha(h)(y)$ , then $x\sim_{R_Q} x'$ and 
 $$
 y\sim_{R_Q} y' = \theta_{g_{n,j}} (x') \sim_{R} x' \sim_{R_Q} y.$$  Since $\theta$ generates $R_Q$, it follows that $R_Q\subset R$ and thus $y\sim_R x$.  But this means that $R=R_Q$.
 \end{proof}
 
 \begin{lemma} 
The graphing $\Gamma'$ is a treeing and $C(\Gamma') = 1+\sum_{r=0}^\infty  {k_r}{|Q_r|^{-1}}.$
\end{lemma}
\begin{proof} 
Let $\theta_n$ be the restriction of $\theta$ to $\bigcup \{ F(j_1,\dots,j_n) :  j_r < q_r \textrm{ for some } r \leq n\}$.  It is not hard to see, arguing just as above, that $\theta_n$ generates the partition $R_n$, and moreover that the graphing
 $$ 
  \Gamma'_n = \{ \theta_n, \phi_{r,j} : 0\leq r\leq n, 1\leq j\leq k_r\}$$ generates $R_{G_n}$, the equivalence relation of being in the same orbit of the restriction of the action $\alpha$ to $G_n$.  
 Let us suppose that for some non-trivial word $w=w_1\cdots w_k$ with $w_j \in \Gamma'$, we have that $w(x) = x$ for all $x$ in some non-null set $F$.  By shrinking $F$ we may assume that there exists some integer $n$  so that  $$F, w_k(F), w_{k-1}w_k(F),\dots, w_1\cdots w_k(F) \subset \bigcup \{ F(j_1,\dots,j_n) : \exists r \leq n \textrm{ s.t. } j_r <q_r\}.$$  
It follows that the restriction of $w$ to $F$ is actually a word in the graphing $\Gamma'_n$.  Thus it is sufficient to prove that $\Gamma'_n$ is a treeing, as we would then be able to conclude that $w$ is (modulo cancellations) the trivial word.

 To prove that $\Gamma'_n$ is a treeing, it is sufficient to show that 
  $C(\Gamma') = C(G_n)$ (\cite{gaboriau:cost}).   This is clear when $n=0$, since in that case $C(\Gamma')=k_0 = C(G_0)$.  Assuming that equality $C(\Gamma'_n)=C(G_n)$ holds, we compute:
 \begin{eqnarray*}
  C(\Gamma'_{n+1})&=&(1- (q_0\cdots  q_n q_{n+1})^{-1}) + \sum_{r=0}^{n+1} k_r \cdot (q_0 \cdots q_r)^{-1}\\ 
  &=& (1 - |Q_{n+1}|^{-1}) + \sum_{r=0}^{n+1} k_r |Q_r|^{-1}  \\
  &=& C(\Gamma'_{n}) - (1- |Q_{n}|^{-1} ) + (1- |Q_{n+1}|^{-1}) +k_{n+1} |Q_{n+1}|^{-1} \\
  &=& C(\Gamma'_{n}) - C(Q_{n}) + 1 + ( k_{n+1} -1) |Q_{n+1}|^{-1}  \\
  &=& C(G_n) - C(Q_{n}) + C(Q_{n+1} \times \mathbb{F}_{k_{n+1}}) \\
  &=& C(G_{n+1})
  \end{eqnarray*}  because of the inductive assumption, the fact that  $G_{n+1} = G_{n}*_{Q_{n}} (Q_{n+1} \times \mathbb{F}_{k_n})$, and Gaboriau's formula for cost of an amalgamated free product \cite{gaboriau:cost}. 
  
 The cost of $\Gamma'$ can be computed in a similar way:
 $$
 C(\Gamma') = 1 + \sum_{r=0}^\infty k_r \cdot (q_0 \cdots q_r)^{-1} = 1 + \sum_{r=0}^\infty k_r |Q_r|^{-1},$$ which completes the proof.
\end{proof}
 
 \begin{corollary} $G$ is a strongly treeable group of cost
  $C(G)= 1 + \sum_{r=0}^\infty k_r |Q_r|^{-1}.$
  \end{corollary}


 The following theorem is a straightforward generalization of a result of Hjorth \cite{hjorth:costAttained}, whose proof we largely follow:
 \begin{theorem} \label{cor:sstreable} 
 Let $R$ be a treeable measure preserving ergodic equivalence relation of cost $t=C(R)$.  
 Choose any $Q_i$, $k_i$ so that $t= \sum_{r=0}^\infty k_r |Q_r|^{-1}$.
 Let $G = G(\{Q_i\}, \{k_i\})$. Then there exists a measure preserving action of $G$ so that $R=R_G$, the equivalence relation of being in the same $G$-orbit.  Thus any such group is OE-free.
 \end{theorem}
 \begin{proof}
 We have seen that if $R = R_G$ for some action of $G$, then the equivalence relation $R$ has the following properties (see Lemma~\ref{lem:partitions} and Lemma~\ref{lem:graphing}).  
\begin{enumerate}\renewcommand{\theenumi}{\alph{enumi}}
\item $R$ is generated by an increasing family of finite equivalence relations $R_n$ and a family of partial morphisms $\phi_{n,j}$.  
\item There exist partitions $P_n = (F(j_1,\dots,j_n): j_k \in \{1,\dots,q_n\}$ are fundamental domains for $R_n$, where  $q_n = |Q_n/Q_{n-1}|$.
\item The domain of  $\phi_{n,j}$ is $F(\bbone_n)$.  
\item The equivalence relation $R_Q :=\bigcup R_n$ is generated by a free action of $\mathbb Z$ acting by an authomorphism $\theta:X\to X$ and furhermore $\Gamma :=\{\theta\} \cup \{ \phi_{n,j}\}$ is a treeing for $R$.  
\end{enumerate}

We now note that, conversely, if we are given an equivalence relation with such structure, then $R=R_G$ for some free action of $G$.  Indeed, we first can consistently choose actions of $Q_j$ in such a way that $R_n = R_{Q_n}$; this defines the action $\alpha$ of $Q = \bigcup Q_j$ on $X$.
Denote by $g_{n,j}$ the generators of $\mathbb{F}_{k_n} \subset G$, we define $\alpha(g_{n,j})$ by requiring that its restriction to $F(\bbone_n)$ is $\theta_{n,j}$, and that it commutes with $\alpha(Q_j)$.  Since the images $F(\bbone_n)$ under $\alpha(Q_j)$ are disjoint and cover all of $X$, this defines the action of $g_{n,j}$ and thus of $\mathbb{F}_{k_n}$.  By assumed commutativity, we thus obtain an action of $Q_n \times \mathbb{F}_{k_n}$.  But passing to the inductive limit, we thus construct an action defined on all of $G$.  The action is clearly measure-preserving and generates $R$; it only remains to see that it is free.
  
  If an element $e\neq w\in (Q_n \times \mathbb{F}_{k_n}) \subset G$ were to act trivially on some $x\in X$, we could, after conjugating $w$ with $Q_n$, assume that $x\in F(\bbone_n)$, in which case triviality of the action of $w$ would contradict the assertion that $\Gamma$ is a treeing.
  
 It thus remains to prove that an arbitrary ergodic measure-preserving equivalence relation $R$ satisfies the properties (a)--(d) above.  But this is a direct consequence of \cite{hjorth:costAttained}
 (see proof of Lemma 4.3 in that paper).
  \end{proof}

 We refer the reader to \cite{gaboriau:ICM,gromov:ME} for the definition of the notion of measure equivalence of groups (see also 4.2 hereafter).  We only state the following easy corollary, which is a measure equivalence analog of Corollary~\ref{cor:WE}:

 \begin{corollary}\label{cor:ME}
 Let $G = G(\{Q_i\}, \{k_i\})$, $G'=G(\{Q'_i\},\{k'_i\})$, and set $t = \sum_r k_r |Q_r|^{-1}$, $t'=\sum_r k'_r |Q'_r|^{-1}$.   Then: \\
 (i) if $t=t'$, any ergodic measure-preserving action of $G'$ is orbit-equivalent to some action of $G$;\\
 (ii) both $t,t'$ are finite or both are infinite iff $G'$ and $G$ are measure-equivalent;\\
 (iii) for any free ergodic measure-preserving action of $G$, there exists a measure-preserving action of $G'$ so that $(L^\infty(X)\rtimes G) \simeq {L^\infty(Y)\rtimes G'} ^{ (t-1)/(t'-1)}$, i.e., $R_G$ is the amplification of $R_{G'}$ by $(t-1)/(t'-1)$. 
 \end{corollary}
\begin{proof} Parts (i) and (iii) follow direction from Theorem~\ref{cor:sstreable} and Gaboriau's induction formula for cost.  For part (ii) note that two groups are measure equivalent iff for some pair of their actions, the associated orbit-equivalence relations are amplifications of one another.
\end{proof}

 \section{Further remarks}
 
We discuss in this section some further properties of the  groups we constructed and make some conjectures, which we relate to open questions formulated in \cite[Remark 4.5]{popa:VC}.  Recall (see \S\ref{remark:Gt}) that we denote by $G_t$ any of the groups $G(\{Q_i\},\{k_i\})$ with $Q=\bigcup Q_i$ ICC and $t=1+\sum k_i |Q_i|^{-1}$, so that, as we have shown in Section 2, $L(G_t)\cong L(\mathbb{F}_t)$.
 
\subsection{Virtual $\text{\rm W}^*$-equivalence and $\text{\rm W}^*$-subordination}  As in many recent papers 
concerning group von Neumann algebras, 
we view an isomorphism $L(\Gamma)\simeq L(\Lambda)$ as a W$^*$-{\it equivalence} between the groups $\Gamma, \Lambda$ and 
write this as $\Gamma \sim_{\text{\rm W}^*} \Lambda$. Also,  
we denote  by W$^*(\Lambda)$ the class of groups $\Gamma$ with the property that $L(\Gamma)\simeq L(\Lambda)$.

Following \cite[ Notation 4.3]{popa:VC}, we denote W$^*_{leq}(\Lambda)$ the class of groups $\Gamma$ with the property that $L(\Gamma)$ can be embedded 
into $L(\Lambda)$ (not necessarily unitally), and use the notation $\Gamma \leq_{\text{\rm W}^*} \Lambda$ for this subordination relation.  

If  $\Gamma, \Lambda$ are ICC groups such that there exists a finite index Hilbert bimodule between $L(\Gamma)$ and $L(\Lambda)$ (equivalently, 
if $L(\Gamma)$ can be embedded as subfactor of finite index in a finite amplification of $L(\Lambda)$, see \cite[ Section 1.4]{popa-Corr}, 
then we say that $\Gamma, \Lambda$ are {\it virtually} W$^*$-{\it equivalent} 
and write $\Gamma \sim_{\text{\rm W}_v^*} \Lambda$, and denote by $ {\rm W}_v^*(\Lambda)$  the virtual W$^*$-equivalence class of $\Lambda$. 

Note that the equivalence between II$_1$ factors $N, M$ requiring the existence of a finite index Hilbert $N-M$ bimodule was 
introduced in  \cite[ Definition 1.4.3]{popa-Corr}, where it is called {\it weak stable equivalence}. But  {\it virtual isomorphism} seems a more suitable terminology for this equivalence 
relation, and so we propose to use it instead.

\subsection{Virtual W$^*$-equivalence versus ME} The analogue of W$^*$-equivalence in measured group theory 
is the notion of {\it orbit equivalence} (OE) of groups: $\Gamma, \Lambda$ are OE if there exists free pmp actions $\Gamma \curvearrowright X$, $\Lambda \curvearrowright Y$ 
and $\Delta: X \simeq Y$ such that $\theta(\Gamma x) =\Lambda \theta(x)$, $\forall_{ae}  \ x\in X$. 

The analogue of virtual W$^*$-equivalence is the notion of {\it measure equivalence} (ME), introduced 
in \cite{gromov:ME}: two groups $\Gamma, \Lambda$ are ME if there exist commuting, measure preserving, free actions of $\Gamma, \Lambda$ 
on some Lebesgue measure space $(\Omega, m)$ such that each action admits a finite fundamental domain. This is  equivalent to 
the fact that there exist free pmp actions 
$\Gamma \curvearrowright X$ and $\Lambda \curvearrowright Y$ and a measure preserving $\Delta: X \rightarrow Y$ which implements a ``virtual orbit equivalence'' 
between the two actions. This condition can be reformulated as  the existence of a virtual isomorphism 
between $L^\infty(X)\rtimes \Gamma$ and $L^\infty(Y)\rtimes \Lambda$ via a finite index Hilbert bimodule that satisfies appropriate conditions on the  
corresponding Cartan subalgebras. 

The similarity between orbit equivalence and W$^*$-equivalence for groups, as well as results 
from deformation-rigidity theory (\cite{popa:strongrig1}), led the second named author to speculate some ten years ago that there may be a correlation between these 
two notions, which may perhaps even coincide.  This has been disproved by Chifan-Ioana who showed in \cite{chifan-ioana} that the ICC groups 
$\Gamma=(S_\infty \times \mathbb F_2) * \mathbb Z$ and $\Lambda=(\mathbb Z \times \mathbb F_2) * \mathbb Z$ are OE but not W$^*$E. 
We note below that these groups are not even W$^*$-subordinated (thus not virtually W$^*$-equivalent), and also provide another example with similar behavior.

 \begin{proposition} \label{prop:ME/virtualW*} $(a)$ The groups $\Gamma = (S_\infty \times \mathbb F_2)* \mathbb Z$, $\Lambda = (\mathbb Z \times \mathbb F_2)*\mathbb Z$ are 
 OE-equivalent but 
 not virtually $\text{\rm W}^*$-equivalent, in fact we have $\Gamma \not\leq_{\text{\rm W}^*} \Lambda$. 
 
 $(b)$ For each $2\leq n \leq \infty$, the groups $S_\infty \wr \mathbb F_n$ and $\mathbb Z \wr \mathbb F_n$ are OE-equivalent  
 but not virtually $\text{\rm W}^*$-equivalent. Moreover, for any $2\leq n, m \leq \infty$, 
 we have $S_\infty \wr \mathbb F_m \not\leq_{\text{\rm W}^*} \mathbb Z\wr \mathbb F_n$. 
 \end{proposition}
 \begin{proof} $(a)$ The fact that $\Gamma \sim_{OE} \Lambda$ was shown in \cite{chifan-ioana}, where it is also shown that if 
 $L(\Gamma) \subset L(\Lambda)$, then up to conjugation by some unitary in $L(\Lambda)$ one may assume $L(S_\infty \times \mathbb F_2)\subset L(\mathbb Z \times \mathbb F_2)=P$. 
 But by the dichotomy result in \cite{popa-vaes:F2} this implies $L(S_\infty) \prec_P L(\mathbb Z)$ (in the sense  [\cite{popa:strongrig1} Section 5]), a contradiction. 
 
  $(b)$ By Dye's theorem, any free ergodic pmp actions of $S_\infty$ and $\mathbb Z$ on $[0,1]$ are OE, via some pmp isomorphism $\Delta$. 
  Taking the product of $\Delta$ ``$\mathbb F_n$-many times'' one gets an isomorphism $\tilde{\Delta}: [0,1]^{\mathbb F_n} \simeq [0,1]^{\mathbb F_n}$ 
  that commutes with the Bernoulli $\mathbb F_n$ action, thus establishing an OE between $S_\infty \wr \mathbb F_n \curvearrowright [0,1]^{\mathbb F_n}$ 
  and $\mathbb Z \wr \mathbb F_n \curvearrowright [0,1]^{\mathbb F_n}$. 
  
On the other hand, if  $L(S_\infty \wr \mathbb F_m)\subset L(\mathbb Z\wr \mathbb F_n)$, then the dichotomy result in \cite{popa-vaes:F2} implies again 
  that $L(S_\infty)\prec L(\mathbb Z)$, a contradiction.  
  \end{proof}

One should point out that there are still no examples of W$^*$-equivalent groups that are not orbit equivalent. 

While the above result  shows that OE and W$^*$E are in general different, these equivalences may still coincide  
for certain restricted classes of groups (as we'll discuss below). It may in fact be more appropriate to seek such a 
correlation between ME and virtual W$^*$-equivalence.

 \subsection{$\text{\rm W}^*$-free groups} As mentioned before, we'll say that an ICC group $\Gamma$ is W$^*$-{\it free}, if $L(\Gamma)\simeq L(\mathbb F_t)$, for some $1< t \leq \infty$. If free group factors are non-isomorphic (a fact known to be 
equivalent to $L(\mathbb F_n)\not\simeq L(\mathbb F_m)$, for some $2\leq n < m \leq \infty$, by \cite{dykema:interpolateddef}, \cite{radulescu:inv}), one would call the corresponding $t$ the W$^*$-{\it free rank} of $\Gamma$ (thus, $\Gamma \sim_{\text{\rm W}^*} G_t$, where $G_t$ is given by Theorem 1.1). We also denote the class of these 
groups by W$^*(\mathbb F_t)$. 

It seems likely that all non-amenable ICC groups W$^*$-subordinated to a free group are W$^*$-free. More precisely, that for any $2\leq n \leq \infty$ we have 
ICC $\cap$ W$^*_{leq}(\mathbb F_n)=\{$amenable ICC$\} \cup \cup_{t>1} \text{\rm W}^*(\mathbb F_t) \cup \text{\rm W}^*(\mathbb F_\infty)$. 
Also, one should have W$^*_v(\mathbb F_\infty)=\text{\rm W}^*(\mathbb F_\infty)$, while W$^*_v(\mathbb F_n)$ should coincide with 
$\cup_{t>1} \text{\rm W}^*(\mathbb F_t)$ (or some remarkable subset of it!), for all $n\geq 2$ finite. 

Moreover, if  a group factor $L(\Gamma)$ can be embedded unitally into $L(G_t)\simeq L(\mathbb F_t)$ with finite Jones index $\alpha$, then one would expect that 
$\Gamma$ would have W$^*$-free rank equal to $1+\alpha(t-1)$ (Nielsen-Schreier type formula). 
More generally,  if $\Gamma, \Lambda$ have W$^*$-free rank $t$ respectively $s$ 
and are virtually W$^*$-equivalent via a finite index 
Hilbert $L(\Gamma)-L(\Lambda)$ bimodule $\mathcal H$ with $r=\dim \mathcal H_{L(\Lambda)}$, $l=\dim_{L(\Gamma)}\mathcal H$, then one would expect $t=l(s-1)/r+1$.  

One should note that if $L(\mathbb F_t)$ are non-isomorphic, then it is sufficient to assume the Nielsen-Schreier type formula for irreducible subfactors 
(or finite index irreducible intertwining Hilbert bimodules) to infer that it holds true for all subfactors (resp. bimodules). Moreover, in order to have 
consistency between the amplification and index formulas, it is necessary 
that for any finite $t>1$ all finite index subfactors of $L(G_t)\simeq L(\mathbb F_t)$ must be extremal. This is of course not the case for $L(\mathbb F_\infty)$, 
which has fundamental group equal to $\mathbb R_{>0}$ by \cite{radulescu:fundgrp}, and thus has non-extremal subfactors of any index $>4$ 
(in fact, by \cite{popa-shlyakhtenko:universal} the standard invariant of any subfactor can be realized as standard invariant of a subfactor of $L(\mathbb F_\infty)$).

\subsection{Alternative characterizations of sub-W$^*$-freeness} 
Recall from \cite[Definition 1.4]{popa:VC}  that a group $\Gamma$ has the {\it universal vanishing 
cohomology} property $\mathcal V\mathcal C$ if any free cocycle action of $\Gamma$ on any II$_1$ factor untwists. Also, given a group $\Lambda$, 
one denotes by ME$_{leq}(\Lambda)$ the class of all groups $\Gamma$  
that have a free  pmp action which can be realized as a sub equivalence 
relation of a  free ergodic pmp $\Lambda$-action. It has already been 
speculated in \cite[Remarks 4.5]{popa:VC} that the 
three classes $\text{\rm W}^*_{leq}(\mathbb F_2)$, ME$_{leq}(\mathbb F_2)$ and $\mathcal V\mathcal C$ may coincide. 
This is motivated by the observation that $\mathcal V\mathcal C \subset \text{\rm W}^*_{leq}(\mathbb F_2)$ and that one already knows 
that ME$_{leq}(\mathbb F_2)$ is equal to $\cup_{n\geq 1}\text{\rm ME}(\mathbb F_n) \cup \text{\rm ME}(\mathbb F_\infty)$ (cf. \cite{gaboriau:cost} and \cite{hjorth:costAttained}). 

On the other hand, this latter class of groups coincides with the class of all treeable groups, which conjecturally  coincides with the class of strongly treeable (or even 
OE-free) groups. 
Thus, the W$^*$-free groups should coincide with the non-amenable ICC treeable groups and with the non-amenable ICC groups in the class $\mathcal V\mathcal C$. 

One should also note that by general properties of the free group factors $L(\mathbb F_n)$, any non-amenable group $\Gamma\leq_{\text{\rm W}^*} \mathbb F_n$ 
must have finite radical, be non-inner amenable, have Haagerup property and Cowling-Haagerup constant equal to $1$. Also, it 
should not contain any infinite subgroup with non-amenable centralizer, and no infinite amenable subgroup with non-amenable normalizer. 
Moreover, all such groups should be exact (see \cite{brown-ozawa:book}).   
 
  \subsection{Some properties of the groups $G_t$} As we have seen,  the groups $\{G_t\}_{t>1}$ give rise to group factors (respectively 
  group measure space factors) with the property that all their amplifications are themselves group factors (respectively group measure space factors).  
 Also, since $G=PSL(2,\mathbb Z)$ is itself a W$^*$-free group of rank $7/6$ (see \cite{harpe-picou:1993}, \cite{dykema:interpolated}), all amplifications of $L(G)$ are 
 group factors. 
 
This is in sharp contrast with certain phenomena in deformation-rigidity theory, where one could  obtain large classes of ICC groups (resp. free ergodic pmp group actions)  
whose associated II$_1$ factors have all (non-trivial) amplifications not representable as group factors (resp. group measure space factors). For instance, 
one shows in \cite{ioana-popa-vaes}  that any non-amenable group $G$ has a canonical ``ICC augmentation'' $\tilde{G}$ such that $L(\tilde{G})^t$ cannot be written as a group factor if $t\neq 1$.  
Also,  \cite{ioana-popa-vaes}, \cite{popa-vaes:F2} provides many examples of free ergodic pmp group actions $G\curvearrowright X$ such that 
no amplification by $t\neq 1$ of $L^\infty(X)\rtimes G$ can be represented as a group measure space factor (e.g., Bernoulli actions 
of $G=\mathbb F_2 \times \mathbb F_2$). Along these same lines, note that the strong form of Connes rigidity conjecture predicts that 
if $G=PSL(n, \mathbb Z)$, $n\geq 3$, then $L(G)^t$ cannot be represented as a group factor, $\forall t\neq 1$. 

Let us also point out that while in Corollary 3.6 one shows that the W$^*$-free groups $G_t$ 
are all treeable, we do not know how to prove that any W$^*$-free group is treeable. Similarly, we could not prove 
  that all treeable groups are in $\mathcal V\mathcal C$. We'll however note below that the groups $G_t$ (in fact all groups 
  considered Section 3)  have the property $\mathcal V\mathcal C$. Some other properties of the W$^*$-free groups $G_t$ will follow from general considerations.

 \begin{proposition} \label{prop:VCproperty}  $(a)$ Assume $\{H_n\}_{n\geq 0} \subset \mathcal V\mathcal C$ is a sequence of groups 
 with finite subgroups $K_n \subset H_n$ such that $K_n \subset K_{n+1}$, $\forall n\geq 0$. Define recursively $G_0=H_0$ and 
 $G_{n+1} = G_n *_{K_{n}} H_{n+1}$. Then $G=\cup_n G_n$ lies in $\mathcal V\mathcal C$. In particular,  all the $\text{\rm W}^*$-free groups $G_t$, $t>1$,  
 have the universal vanishing cohomology property. 
 
 $(b)$ The  $\text{\rm W}^*$-free groups $G_t$, $t>1$, are sofic and Cartan rigid $($i.e., all group measure space factors from free pmp $G_t$-actions have unique 
 Cartan subalgebra up to unitary cojugacy$)$.  

 \end{proposition}
 \begin{proof} The proof of \cite[Proposition 1.5.1$^\circ$]{popa:VC} applies exactly same way to get part $(a)$ of the statement. 
 Then Cartan rigidity of $G_t$ in part $(b)$ follows from the fact $G_t$ are ME to free groups and  \cite[Theorem 11.3]{popa-vaes:II}. The soficity 
 is a consequence of the way $G_t$ are constructed, as a limit of residually finite groups. 
  \end{proof}


\begin{thebibliography}{9}

\bibitem[Bo09]{bowen:2009} L. Bowen, {\em Stable orbit equivalence of Bernoulli shifts over free groups}, {Groups Geom. Dyn.}{\bf 5} (2011), 17-38.

\bibitem[BrO08]{brown-ozawa:book} N. Brown, N. Ozawa,  ``$C^*$-algebras and finite dimensional-approximations'', Amer. Math. Soc. 
Grad. Studies in Math. {\bf 88}, 2008. 

\bibitem[ChI09]{chifan-ioana} I. Chifan, A. Ioana,  {\em On a question of D. Shlyakhtenko}, 
Proc. AMS {\bf 139} (2011), 1091-1093.

\bibitem[C75]{connes:FundThm} A. Connes, {\em Classification of injective factors}, Ann. of Math., {\bf 104} (1976), 73-115.

\bibitem[CS05]{cs:l2betti} A. Connes, D. Shlyakhtenko, {\em $L^2$-homology for von Neumann algebras}, J. Reine Angew. Math. {\bf 586} (2005) 125--168


 \bibitem[D59]{dye:59} H. Dye, {\em On groups of measure preserving transformations}, 
 


\bibitem[Dyk93]{dykema:interpolated} K. Dykema, {\em Free products of hyperfinite von Neumann algebras and free dimension},
 Duke Math. J. {\bf 69} (1993) 97--119
 
 \bibitem[Dyk94]{dykema:interpolateddef} K. Dykema, {\em Interpolated free group factors},
 Pacific J. Math {\bf 164} (1994) 123--135
 
 \bibitem[DR00]{dykema-rad:compressions} K. Dykema, F. Radulescu, {\em Compressions of free products of von Neumann 
 algebras}, Math. Ann. {\bf 316} (2000), 61-82. 
 
\bibitem[Gab00]{gaboriau:cost} D. Gaboriau, {\em Co\^ut des relations d'\'equivalence et des groupes}, Inv. Math. {\bf 139} (2000) 41--98

\bibitem[Gab10]{gaboriau:ICM} D. Gaboriau, {\em Orbit equivalence and measured group theory}. In Proceedings of ICM (Hyderabad, India, 2010), Vol. III, Hindustan Book Agency (2010), 1501-1527.

\bibitem[Gro93]{gromov:ME} M. Gromov, {\em Asymptotic invariants of infinite groups}. In Geometric group theory, Vol. 2 (Sussex, 1991), London Math. Soc. Lecture Note Ser. {\bf 182}, pp. 1--295. Cambridge Univ. Press, Cambridge, 1993.

\bibitem[HV93]{harpe-picou:1993} P. de la Harpe, D. Voiculescu, {\em A problem on the} II$_1$ {\em factors of fuchsian groups}, 
In ÒRecent advances in operator algebras (Orl\'e ans, 1992)Ó, Ast\' erisque 232 (S.M.F. 1995) 155-158.
 
\bibitem[Hjo06]{hjorth:costAttained} G. Hjorth, {\em A lemma for cost attained}, Annals of Pure and Applied Logic {\bf 143} (2006) 87--102

\bibitem[IPV10]{ioana-popa-vaes} A. Ioana, S. Popa, S. Vaes, {\em A Class of superrigid group von Neumann algebras}, Annals of Math. {\bf 178} (2013), 231-286 
(math.OA/1007.1412).


\bibitem[KM04]{kechris-miller:OE} A. Kechris, B. Miller, ``Topics in orbit equivalence'', Lecture Notes in Mathematics, vol. 1852, Springer-Verlag, Berlin, 2004.

\bibitem[MvN36]{MvN36} F. Murray, J. von Neumann, {\em On rings of operators}, Ann. Math. {\bf 37} (1936), 116-229.


\bibitem[MvN43]{MvN43}  F. Murray, J. von Neumann, {\em Rings of operators IV}, Ann. Math. {\bf 44} (1943), 716-808.


\bibitem[P86]{popa-Corr} S. Popa, {\em Correspondences}, Preprint 1986 http://www.math.ucla.edu/~popa/preprints.html

\bibitem[P06]{popa:strongrig1} S. Popa,  {\it Deformation and rigidity for group actions
and von Neumann algebras}, in ``Proceedings of the International
Congress of Mathematicians'' (Madrid 2006), Volume I, EMS Publishing House,
Zurich 2006/2007, pp. 445-479.

\bibitem[P18]{popa:VC} S. Popa, {\em On the vanishing cohomology problem for cocycle actions of groups on}  II$_1$ {\em factors}, 
math.OA/1802.09964

\bibitem[PS03]{popa-shlyakhtenko:universal}  {\em Universal properties of $L(\mathbb{F}(\infty))$ in subfactor theory}, Acta Math. {\bf 191} (2003), no. 2, 225--257.

\bibitem[PV11]{popa-vaes:F2}S. Popa, S. Vaes, {\em Unique Cartan decomposition for} II$_1$ 
{\em factors arising from arbitrary actions of free groups}, Acta Mathematica, {\bf 194} (2014), 237-284 (math.OA/1111.6951)


\bibitem[PV12]{popa-vaes:II}S. Popa, S. Vaes, {\em Unique Cartan decomposition for} II$_1$ 
{\em factors arising from arbitrary actions of hyperbolic groups}, J. Reine Angew. Math., {\bf 690} (2014), 433-458.


\bibitem[Rad92]{radulescu:fundgrp} F. R\u{a}dulescu, {\em The fundamental group of the von Neumann algebra of a free group with infinitely many generators}, J. Amer. Math. Soc. {\bf 5} (1992) 517--532.

\bibitem[Rad94]{radulescu:inv} F. R\u{a}dulescu, {\em Random matrices, amalgamated free products and subfactors of the von Neumann algebra of a free group, of non-integer index}, Inv. Math. {\bf 115} (1994) 347--389 


\bibitem[Shl98]{shlyakht:amalg} D. Shlyakhtenko, {\em Some applications of freeness with amalgamation}, J. Reine Angew. Math. {\bf 500} (1998) 191--212

\bibitem[Shl99]{shlyakht:Avalued}  D. Shlyakhtenko, {\em $A$-valued semicircular systems}, J. Funct. Anal. {\bf 166}  (1999) 1--47


\bibitem[Vo89]{voiculescu:89} D. Voiculescu, {\em Circular and semicircular systems and free product factors}, 
In ``Operator algebras, unitary representations, enveloping algebras, and invariant theory'' (Paris, 1989), 
Progr. Math. {\bf 92}, Birkh\"{a}user, Boston, 1990, pp. 45-60. 



\end{thebibliography}
 \end{document}